%% file: main.tex
\documentclass[11pt,reqno]{amsart}

\newif\ifjournalversion
\journalversionfalse

\usepackage[T1]{fontenc}
\usepackage{microtype}
\usepackage{mathtools,amssymb}
\usepackage{enumitem}
\usepackage{booktabs}
\usepackage[hidelinks]{hyperref}
\usepackage[nameinlink,noabbrev]{cleveref}

\hypersetup{
  pdftitle={Log-concavity of elementary coefficients for low-rank abelian Hessenberg graphs, with a counterexample in general},
  pdfauthor={Boris Kafidov},
  pdfsubject={An all-orders e-log-concavity theorem for abelian complement-Ferrers rank at most three and a 13-vertex counterexample in general},
  pdfkeywords={chromatic quasisymmetric function, e-log-concavity, Hessenberg function, natural unit interval graph, q-hit number}
}

\setlist[itemize]{leftmargin=*,topsep=4pt,itemsep=2pt}
\setlist[enumerate]{leftmargin=*,topsep=4pt,itemsep=2pt}
\allowdisplaybreaks[2]
\numberwithin{equation}{section}

\newtheorem{theorem}{Theorem}[section]

\newtheorem{proposition}[theorem]{Proposition}
\newtheorem{lemma}[theorem]{Lemma}
\theoremstyle{definition}
\newtheorem{definition}[theorem]{Definition}

\theoremstyle{remark}

\DeclareMathOperator{\asc}{asc}
\newcommand{\xx}{\mathbf{x}}
\newcommand{\N}{\mathbb{N}}
\newcommand{\F}{\mathcal F}
\newcommand{\coeff}[2]{[#1]#2}

\title[Low-rank $e$-log-concavity and a counterexample]
      {Log-concavity of elementary coefficients for low-rank abelian
       Hessenberg graphs, with a counterexample in general}
\author{Boris Kafidov}
\thanks{ORCID: \href{https://orcid.org/0009-0009-3146-3781}%
  {0009-0009-3146-3781}.}
\address{University of Toronto, Toronto, Ontario, Canada}
\email{b.kafidov@mail.utoronto.ca}
\date{31 July 2026}

\subjclass[2020]{Primary 05E05; Secondary 05C15, 05A20}
\keywords{Chromatic quasisymmetric function, elementary symmetric function,
Hessenberg function, natural unit interval graph, log-concavity, $q$-hit number}

\begin{document}

\begin{abstract}
Let $X_{G_h}(\xx;q)=\sum_{\mu\vdash n}c_\mu(q)e_\mu(\xx)$ be the
chromatic quasisymmetric function of the natural unit interval graph attached
to a Hessenberg function $h$.  We establish an infinite class, valid in all
orders, for which every nonzero polynomial $c_\mu(q)$ has a nonnegative,
log-concave coefficient sequence with interval support.  Namely, this holds
whenever $h$ is abelian and its
complement-Ferrers partition $\lambda$ satisfies
$\min\{\lambda_1,\ell(\lambda)\}\leq3$; equivalently, the diagram has at most
three rows or at most three columns.  Cubic interpolation reduces the
rank-three case to a uniform theorem for a difference of two products of
four $q$-integers, proved by positive decomposition, interval methods, and
finite-window smoothing.  The argument also yields explicit formulas for
every supported elementary coefficient in complement-Ferrers rank at most
three.  We also include a connected
13-vertex natural unit interval graph for which one elementary coefficient is
positive, palindromic, and unimodal but not log-concave, thereby recording the
failure of the unrestricted conjecture.  Thus low complement-Ferrers rank
gives a substantial positive regime even though coefficientwise
$e$-log-concavity fails in general.
\end{abstract}

\maketitle

\input{sections/introduction}
\input{sections/preliminaries}
\input{sections/rank_three_reduction}
\input{sections/sequence_lemmas}
\input{sections/quartic_primitive}
\input{sections/quartic_prefix}
\input{sections/rank_three_completion}
\input{sections/lower_rank}
\input{sections/counterexample}
\input{sections/verification_outlook}

\appendix
\input{sections/appendix_counterexample}

\bibliographystyle{amsplain}
\bibliography{references}

\end{document}

%% file: sections/introduction.tex
\section{Introduction}\label{sec:introduction}

For a natural unit interval graph $G$, the chromatic quasisymmetric function
$X_G(\xx;q)$ is symmetric in $\xx$ and has an elementary-basis expansion
\begin{equation}\label{eq:intro-expansion}
  X_G(\xx;q)=\sum_{\mu\vdash |V(G)|}c_\mu(q)e_\mu(\xx).
\end{equation}
The refined Shareshian--Wachs conjecture predicts that every $c_\mu(q)$ has
nonnegative, unimodal coefficients
\cite[Conjecture~1.3]{ShareshianWachs}.  Hikita recently proved the
Stanley--Stembridge conjecture.  More generally, he showed that
$c_\mu(q)\geq0$ for every real $q>0$
\cite[Corollaries~1.7 and~1.8]{Hikita}; coefficientwise positivity in
$\mathbb Z_{\geq0}[q]$ remains open.  Motivated by the $q$-hit-number formulas
in the abelian case, Abreu and Nigro conjectured the strictly stronger
assertion that every $c_\mu(q)$ is log-concave
\cite[Conjecture~5.3]{AbreuNigro}.  Kiem and Lee asked the corresponding
question for generalized Hessenberg functions
\cite[Question~4.10]{KiemLee}.  The same assertion for natural unit interval
graphs was subsequently formulated by Tom
\cite[Conjecture~5.6]{TomSigned} and by Sagan and Tom
\cite[Conjecture~4.2]{SaganTom}.

Several weaker or more specialized coefficientwise results are known.
Harada and Precup proved graded $e$-positivity in the abelian case
\cite[Corollary~7.26]{HaradaPrecup}; Cho and Huh gave a combinatorial proof
and explicit formulas for a subfamily
\cite[Theorems~3.3 and~4.2]{ChoHuh}; and Abreu and Nigro proved
$e$-positivity and $e$-unimodality for every abelian Hessenberg function
\cite[Proposition~3.6]{AbreuNigro}.  Colmenarejo, Morales, and Panova gave an
elementary treatment of the relevant $q$-hit expansions and compared their
normalizations
\cite[Theorems~1.2 and~1.3 and Proposition~B.1]{ColmenarejoMoralesPanova}.
Related manifestly positive formulas for abelian elementary coefficients were
obtained independently by Lee and Soh
\cite[Theorem~23 and Corollary~25]{LeeSoh}.
Previously established log-concavity cases include paths
\cite[Proposition~5.2]{AbreuNigro}, the closed-form family of Cho and Huh
\cite[Theorem~4.2]{ChoHuh}, lollipop graphs
\cite[Proposition~4.4]{HuhNamYoo}, \cite[discussion following
Proposition~5.2]{AbreuNigro}, and graphs obtained by gluing cliques of sizes
$a$ and $b$ at one vertex (denoted $K_{a,b}$ in the cited papers)
\cite[Corollary~4.14]{TomSigned}, \cite[Example~4.3]{SaganTom}.  None of these
results establishes log-concavity for the full low-rank abelian class treated
here.
At $q=1$, Matherne, Morales, and Selover proved that every finite-variable
specialization of the ordinary chromatic symmetric function of an abelian
Dyck path is Lorentzian
\cite[Theorem~6.8]{MatherneMoralesSelover}; that multivariate statement is
distinct from log-concavity of each elementary coefficient as a polynomial
in $q$.

There are two complementary phenomena.  The universal log-concavity
assertion is false: a connected graph on $13$ vertices has an elementary
coefficient whose first three nonzero entries are $1,6,38$.  On the other
hand, log-concavity holds uniformly on a large class naturally singled out
by the abelian $q$-hit theory.  The purpose of this paper is to establish the
positive theorem, explain its algebraic mechanism, and contrast that failure
with the structured positive regime.

Let $h:[n]\to[n]$ be a Hessenberg function and let $\lambda$ be the partition
whose transpose has parts $(\lambda^t)_i=n-h(i)$, with zero parts omitted.
When $h$ is abelian, the complement of the associated natural unit interval
graph is Ferrers bipartite.  We call
\begin{equation}\label{eq:intro-rank}
  \rho(\lambda)=\min\{\lambda_1,\ell(\lambda)\}
\end{equation}
the \emph{complement-Ferrers rank}; set $\rho(\varnothing)=0$.  Our main
result is the following.

\begin{theorem}[Main theorem]\label{thm:main}
Let $h:[n]\to[n]$ be abelian, and let $\lambda$ be its complement-Ferrers
partition.  If $\rho(\lambda)\leq3$, then every nonzero elementary-basis
coefficient $c_\mu(q)$ in \eqref{eq:intro-expansion} is a nonnegative
log-concave polynomial in $q$ with interval support.
\end{theorem}

Equivalently, \cref{thm:main} covers every abelian complement-Ferrers diagram
having at most three rows or at most three columns, in every order and for
every elementary coefficient simultaneously.  The conclusion is stronger
than coefficientwise $e$-positivity and $e$-unimodality, and identifies
complement-Ferrers rank four as the first structural layer not covered by the
theorem.

The proof reduces the genuine rank-three case to one algebraic kernel.  For
$[m]_q=1+q+\cdots+q^{m-1}$, set
\begin{equation}\label{eq:intro-quartic}
  \F_{M;x,y,z}(q)
  =[M]_q[x]_q[y]_q[z]_q
   -[M+3]_q[x-1]_q[y-1]_q[z-1]_q.
\end{equation}
Under the hypotheses of \cref{thm:uniform-quartic}, we prove that
$[L]_q\F_{M;x,y,z}$ is nonnegative and log-concave.  Three ideas enter.
First, a telescoping identity expresses the normalized kernel
as a sum of three products of $q$-integers with a common symmetry center.
Second, a Clebsch--Gordan interval decomposition proves one-crossing control
for the primitive.  Third, cumulative prefixes are handled by six signed
walls, reduced through an odd-interval moment inequality.  A finite-window
smoothing lemma then passes from prefixes to multiplication by every
long-enough $q$-integer.

\ifjournalversion
\else
The following counterexample was first proved in version~1 of this article;
it disproves the unrestricted conjecture.
\fi

\begin{theorem}[Counterexample]\label{thm:intro-counterexample}
For
\[
  h=(2,4,4,6,7,10,10,10,10,12,12,13,13)
\]
and $G=G_h$, one has
\begin{align*}
  \coeff{e_{(6,5,1,1)}}{X_G(\xx;q)}
  ={}&q^5+6q^6+38q^7+128q^8+257q^9+362q^{10}+400q^{11}\\
     &+362q^{12}+257q^{13}+128q^{14}+38q^{15}+6q^{16}+q^{17}.
\end{align*}
Consequently, $X_G(\xx;q)$ is not coefficientwise $e$-log-concave.
\end{theorem}

The coefficient in \cref{thm:intro-counterexample} remains positive,
palindromic, and unimodal, but $6^2<1\cdot38$.  Thus the example separates
log-concavity from the weaker shape properties.  It is nonabelian, so it does
not conflict with \cref{thm:main}.

\medskip\noindent\textbf{Novelty and scope.}
To the best of the author's knowledge, no earlier theorem proves
coefficientwise $e$-log-concavity in every order for all abelian
complement-Ferrers diagrams having at most three rows or at most three
columns.  Sagan and Tom proved coefficientwise nonnegativity in the
$q$-graded setting when $\mu_1\leq2$, and in the ungraded setting when
$\mu_1\leq3$ \cite[Corollaries~2.5 and~2.6]{SaganTom}.  These hypotheses
constrain the partition indexing a coefficient, whereas ours constrains the
graph's complement-Ferrers shape; the quantifiers are different.

\medskip\noindent\textbf{Organization.}
\Cref{sec:preliminaries}--\cref{sec:lower-rank} prove the positive theorem.
The remaining sections establish the counterexample, discuss implications,
and give its exact certificate.

%% file: sections/preliminaries.tex
\section{Preliminaries}\label{sec:preliminaries}

\subsection{Hessenberg graphs and chromatic quasisymmetric functions}

Write $[n]=\{1,\ldots,n\}$ and $\N=\{1,2,\ldots\}$.  A
\emph{Hessenberg function} is a weakly increasing map $h:[n]\to[n]$
satisfying $h(i)\geq i$.  Its natural unit interval graph $G_h$ has vertex
set $[n]$ and an edge $\{i,j\}$, for $i<j$, exactly when $j\leq h(i)$.
For a proper coloring $\kappa:[n]\to\N$, put
\[
  \asc_{G_h}(\kappa)
  =\#\{\{i,j\}\in E(G_h):i<j\text{ and }\kappa(i)<\kappa(j)\}.
\]
The chromatic quasisymmetric function is
\begin{equation}\label{eq:xg-definition}
  X_{G_h}(\xx;q)
  =\sum_{\kappa\ \mathrm{proper}}
    q^{\asc_{G_h}(\kappa)}x_{\kappa(1)}\cdots x_{\kappa(n)}.
\end{equation}
For natural unit interval graphs it is symmetric in $\xx$
\cite[Theorem~4.5]{ShareshianWachs}; hence
\begin{equation}\label{eq:e-expansion}
  X_{G_h}(\xx;q)
  =\sum_{\mu\vdash n}c_\mu(q)e_\mu(\xx).
\end{equation}
For any basis element or monomial $b$, the notation $[b]F$ means the
coefficient of $b$ in $F$.

A finite sequence is \emph{unimodal} if it is weakly increasing up to some
index and weakly decreasing thereafter.  Its \emph{first mode} is the least
index at which its maximum is attained.

\begin{definition}\label{def:log-concavity}
A nonzero polynomial $a(q)=\sum_{k=r}^s a_kq^k$ is \emph{log-concave with
interval support} if $a_k>0$ for $r\leq k\leq s$ and
\[
  a_k^2\geq a_{k-1}a_{k+1}
  \qquad(r<k<s).
\]
We call $X_{G_h}(\xx;q)$ \emph{coefficientwise $e$-log-concave} if every
nonzero $c_\mu(q)$ in \eqref{eq:e-expansion} has this property.
\end{definition}

Zero elementary coefficients are ignored.  This convention avoids assigning
a support to the zero polynomial.

\subsection{Abelian functions and complement-Ferrers rank}

The complete Hessenberg function $(n,\ldots,n)$ is abelian by the empty-ideal
convention.  For a noncomplete Hessenberg function, abelianness is equivalent
to
\begin{equation}\label{eq:abelian-criterion}
  h(h(1)+1)=n
\end{equation}
\cite[Remark~2.10]{AbreuNigro}.  Associate to $h$ the partition $\lambda$
defined, after zero parts are omitted, by
\begin{equation}\label{eq:associated-partition}
  (\lambda^t)_i=n-h(i).
\end{equation}
For a noncomplete $h$, let $k=\min\{i:h(i)=n\}$.  Then
$\lambda_1=k-1$ and $\ell(\lambda)=n-h(1)$.  Consequently,
\begin{equation}\label{eq:abelian-ferrers-criterion}
  h(h(1)+1)=n
  \quad\Longleftrightarrow\quad
  \lambda_1+\ell(\lambda)\leq n.
\end{equation}
Indeed, the left side is equivalent to $k\leq h(1)+1$.  For the complete
function, $\lambda=\varnothing$ and preservation of abelianness under
transposition is immediate.  Thus transposing the complement-Ferrers
partition preserves abelianness.
This is the transpose of the partition convention used in parts of
\cite{AbreuNigro}; all formulas below use \eqref{eq:associated-partition}.
Equivalently, the nontrivial part of the graph complement is Ferrers
bipartite; this is compatible with the two-clique characterization of
abelian Hessenberg graphs in \cite[p.~1061]{HaradaPrecup}.  We use the rank parameter
\begin{equation}\label{eq:ferrers-rank}
  \rho(\lambda)=\min\{\lambda_1,\ell(\lambda)\}.
\end{equation}
For the empty partition, set $\lambda_1=\ell(\lambda)=\rho(\lambda)=0$.

Transposing the Ferrers diagram produces the Hessenberg function
\begin{equation}\label{eq:transpose-h}
  h^t(i)=n+1-\min\{j:h(j)\geq n+1-i\}.
\end{equation}
The transpose symmetry is exact at the same value of $q$:
\begin{equation}\label{eq:transpose-symmetry}
  X_{G_{h^t}}(\xx;q)=X_{G_h}(\xx;q)
\end{equation}
\cite[Theorem~1.2 and Proposition~2.14]{AbreuNigro}.  Thus width and length
of $\lambda$ may be exchanged without changing any elementary coefficient
polynomial.

\subsection{\texorpdfstring{$q$}{q}-integers and the abelian specialization}

We use
\[
  [m]_q=1+q+\cdots+q^{m-1},
  \qquad
  [m]_q!=\prod_{i=1}^m[i]_q,
\]
with $[0]_q=0$ and $[0]_q!=1$.  Define
\begin{equation}\label{eq:phi-definition}
  \varphi_m(\xi)=\prod_{i=0}^{m-1}(\xi-[i]_q).
\end{equation}
The following is the exact form of the abelian specialization needed later.

\begin{proposition}[Abreu--Nigro specialization]
\label{prop:abelian-specialization}
There is a $\mathbb Q(q)$-algebra homomorphism from the ring of symmetric
functions $\Lambda_{\mathbb Q(q)}$ to $\mathbb Q(q)[\xi]$, denoted $\alpha$,
satisfying
\begin{equation}\label{eq:alpha-em}
  \alpha(e_m)=\frac{\varphi_m(\xi)}{[m]_q!}
\end{equation}
and
\begin{equation}\label{eq:alpha-xg}
  \alpha\!\left(X_{G_h}(\xx;q)\right)
  =\prod_{i=1}^n\bigl(\xi-[h(i)-i]_q\bigr).
\end{equation}
\end{proposition}

\begin{proposition}[Abelian support]\label{prop:abelian-support}
If $h$ is abelian, then $c_\mu(q)=0$ unless
\[
  \mu=(n-j,j),\qquad 0\leq j\leq\rho(\lambda).
\]
Here and below, the zero part is omitted when $j=0$.
\end{proposition}

The specialization and product formula are established in the proof of
\cite[Proposition~3.3]{AbreuNigro}; the support statement is the abelian
$q$-hit expansion in the normalization of
\cite[Theorems~1.3 and~4.3]{AbreuNigro}.  We verify the restricted linear
independence needed later directly through a triangular evaluation.
The underlying $q$-hit theory originates in the $q$-rook framework of
Garsia and Remmel \cite{GarsiaRemmel}; see also Haglund's symmetry and
unimodality results \cite[Theorem~6]{HaglundQRook}.  Since common $q$-hit conventions
can differ by powers of $q$, we use the displayed normalization throughout.

\subsection{The uniform quartic theorem}

The analytic core is a uniform statement about \eqref{eq:intro-quartic}.

\begin{theorem}[Uniform quartic theorem]\label{thm:uniform-quartic}
Let $M,x,y,z$ be integers satisfying
\[
  M\geq1,
  \qquad 1\leq x\leq y\leq z\leq M+2,
  \qquad (x,y,z)\neq(M+2,M+2,M+2).
\]
Then, for every integer $L\geq M+1$, the polynomial
\[
  [L]_q\F_{M;x,y,z}(q)
\]
is nonzero and has a nonnegative, log-concave coefficient sequence with
interval support.
\end{theorem}

%% file: sections/rank_three_reduction.tex
\section{The exact rank-three reduction}
\label{sec:rank-three-reduction}

We now reduce the genuine complement-Ferrers rank-three case of
\cref{thm:main} to \cref{thm:uniform-quartic}.  The coefficient extraction
uses the abelian support statement and the specialization from
\cref{prop:abelian-support,prop:abelian-specialization}; in particular, no
inversion formula for the full elementary basis is needed.

\subsection{The width-three orientation}

The width $\lambda_1$ of the partition in
\eqref{eq:associated-partition} is the number of entries of $h$ that are
strictly smaller than $n$.  Transposition exchanges $\lambda_1$ and
$\ell(\lambda)$ and preserves $X_{G_h}(\xx;q)$.  Hence a genuine rank-three
shape may be oriented so that $\lambda_1=3$.

\begin{proposition}[Rank-three normal form]
\label{prop:rank-three-domain}
After transposition if necessary, every abelian Hessenberg function of
complement-Ferrers rank three has the form
\begin{equation}\label{eq:rank-three-h}
  h=(u,v,w,n,\ldots,n),
  \qquad
  3\leq u\leq n-3,
  \qquad
  u\leq v\leq w\leq n-1.
\end{equation}
It is disconnected if and only if
\begin{equation}\label{eq:rank-three-disconnected-h}
  h=(3,3,3,n,\ldots,n).
\end{equation}
In the connected case, put
\begin{equation}\label{eq:rank-three-parameters}
  N=n-4,
  \qquad A=u-1,
  \qquad B=v-2,
  \qquad C=w-3,
  \qquad S=A+B+C.
\end{equation}
Then the exact parameter domain is
\begin{equation}\label{eq:rank-three-domain}
  \begin{gathered}
    N\geq2,
    \qquad 2\leq A\leq N,
    \qquad A-1\leq B\leq N+1,\\
    \max\{1,B-1\}\leq C\leq N.
  \end{gathered}
\end{equation}
Conversely, every quadruple in \eqref{eq:rank-three-domain} determines a
connected abelian Hessenberg function of the form
\eqref{eq:rank-three-h} and complement-Ferrers rank three.
\end{proposition}

\begin{proof}
Width three gives exactly three entries of $h$ smaller than $n$, so
$h=(u,v,w,n,\ldots,n)$.  Since the rank is genuinely three, the height of
the complement-Ferrers diagram is at least three.  By
\eqref{eq:associated-partition}, this height is $n-h(1)=n-u$; hence
$u\leq n-3$.  The abelian criterion \eqref{eq:abelian-criterion} forces
$u\geq3$: if $u\leq2$, then the entry $h(u+1)$ would be one of the first
three, and hence would be smaller than $n$, contradicting
$h(u+1)=h(h(1)+1)=n$.  Monotonicity and the width-three assumption give the
remaining inequalities in \eqref{eq:rank-three-h}.

For a natural unit interval graph, a cut between $i$ and $i+1$ disconnects
the graph precisely when $h(i)=i$.  In \eqref{eq:rank-three-h}, the
inequalities $u\geq3$ and $v\geq u$ rule this out at $i=1,2$, while every
entry from the fourth onward equals $n$.  Thus disconnection occurs exactly
when $w=h(3)=3$, which by monotonicity forces $u=v=w=3$.

Suppose now that the graph is connected.  Then $w\geq4$.  The bounds in
\eqref{eq:rank-three-h}, after applying
\eqref{eq:rank-three-parameters}, become
\[
  2\leq A\leq N,
  \qquad A-1\leq B\leq N+1,
  \qquad 1\leq C\leq N,
  \qquad C\geq B-1,
\]
which is \eqref{eq:rank-three-domain}.  Conversely, given parameters in
that domain, set $u=A+1$, $v=B+2$, and $w=C+3$.  The displayed inequalities
give \eqref{eq:rank-three-h}, with $w\geq4$.  Since $u+1\geq4$, one has
$h(h(1)+1)=h(u+1)=n$, so $h$ is abelian.  Also $u\leq n-3$ gives height at
least three, while exactly three entries are less than $n$.  Thus the rank is
three, and the preceding connectivity criterion applies.
\end{proof}

\subsection{Cubic interpolation}

For a connected function in \cref{prop:rank-three-domain}, the multiset
of area values $h(i)-i$ is
\begin{equation}\label{eq:rank-three-area-multiset}
  \{A,B,C,N,N-1,\ldots,1,0\}.
\end{equation}
By the abelian $q$-hit support statement in
\cref{prop:abelian-support}, only the four elementary coefficients
\begin{equation}\label{eq:rank-three-cj-definition}
  c_j(q)=\coeff{e_{(n-j,j)}}{X_{G_h}(\xx;q)},
  \qquad 0\leq j\leq3,
\end{equation}
can be nonzero.

Apply \eqref{eq:alpha-xg} to \eqref{eq:rank-three-area-multiset} and divide
the resulting identity by the common factor
$\varphi_{n-3}(\xi)=\varphi_{N+1}(\xi)$.  Using \eqref{eq:alpha-em} gives
\begin{align}
&(\xi-[A]_q)(\xi-[B]_q)(\xi-[C]_q)\notag\\
={}&c_0\frac{\prod_{j=1}^3(\xi-[N+j]_q)}{[n]_q!}
 +c_1\frac{\xi(\xi-[N+1]_q)(\xi-[N+2]_q)}{[n-1]_q!}\notag\\
&+c_2\frac{\xi(\xi-1)(\xi-[N+1]_q)}{[n-2]_q![2]_q}
 +c_3\frac{\xi(\xi-1)(\xi-[2]_q)}{[n-3]_q![3]_q!}.
\label{eq:rank-three-cubic}
\end{align}
The evaluations
\begin{equation}\label{eq:rank-three-evaluation-points}
  \xi=0,
  \qquad \xi=1,
  \qquad \xi=[N+1]_q,
  \qquad \xi=[N+2]_q
\end{equation}
produce a block-triangular system in the solve order $(c_0,c_1,c_3,c_2)$.  In
particular, they directly verify the needed linear independence in this
four-dimensional subspace.

\begin{proposition}[Exact rank-three coefficients]
\label{prop:rank-three-coefficients}
For the parameters in \eqref{eq:rank-three-domain}, the coefficients in
\eqref{eq:rank-three-cj-definition} are
\begin{align}
  c_0={}&[A]_q[B]_q[C]_q[n]_q[N]_q!,
  \label{eq:rank-three-c0}\\
  c_1={}&q[N-1]_q![N+2]_q\F_{N;A,B,C},
  \label{eq:rank-three-c1}\\
  c_2={}&q^{S-N-2}[2]_q[N-2]_q![N]_q
  \F_{N-1;N+2-A,N+2-B,N+2-C},
  \label{eq:rank-three-c2}\\
  c_3={}&q^{S-3}[2]_q[3]_q[N-2]_q!
  [N+1-A]_q[N+1-B]_q[N+1-C]_q.
  \label{eq:rank-three-c3}
\end{align}
\end{proposition}

\begin{proof}
We use repeatedly
\begin{equation}\label{eq:q-integer-evaluation-identities}
  1-[r]_q=-q[r-1]_q,
  \qquad
  [r]_q-[s]_q=q^s[r-s]_q
  \quad(r\geq s).
\end{equation}
At $\xi=0$, every term of \eqref{eq:rank-three-cubic} except the $c_0$ term
vanishes.  Cancelling the three negative signs gives
\[
  c_0
  =[A]_q[B]_q[C]_q
    \frac{[n]_q!}{[N+1]_q[N+2]_q[N+3]_q}
  =[A]_q[B]_q[C]_q[n]_q[N]_q!,
\]
which proves \eqref{eq:rank-three-c0}.

At $\xi=1$, the $c_2$ and $c_3$ terms vanish.  Substituting the formula for
$c_0$ and using \eqref{eq:q-integer-evaluation-identities} reduces the
remaining equation to
\begin{align*}
  &\frac{q^3}{[N+3]_q}
  \left(
    [N]_q[A]_q[B]_q[C]_q
    -[N+3]_q[A-1]_q[B-1]_q[C-1]_q
  \right)\\
  &\hspace{35mm}
  =c_1\frac{q^2}{[N+3]_q[N+2]_q[N-1]_q!}.
\end{align*}
The parenthesized expression is $\F_{N;A,B,C}$, proving
\eqref{eq:rank-three-c1}.

At $\xi=[N+1]_q$, the terms involving $c_0,c_1,c_2$ vanish.  The left side
is
\[
  q^S[N+1-A]_q[N+1-B]_q[N+1-C]_q,
\]
whereas the basis factor multiplying $c_3$ is
\[
  \frac{[N+1]_q\,q[N]_q\,q^2[N-1]_q}
       {[N+1]_q![3]_q!}
  =\frac{q^3}{[N-2]_q![3]_q!}.
\]
This is \eqref{eq:rank-three-c3}, since
$[3]_q!=[2]_q[3]_q$.

Finally evaluate at $\xi=[N+2]_q$.  The $c_0$ and $c_1$ terms vanish, and
the left side is
\[
  q^S[N+2-A]_q[N+2-B]_q[N+2-C]_q.
\]
The basis factors multiplying $c_2$ and $c_3$ are, respectively,
\[
  \frac{q^{N+2}}{[2]_q[N]_q!}
  \qquad\text{and}\qquad
  \frac{q^3[N+2]_q}{[N-1]_q![3]_q!}.
\]
Substitute \eqref{eq:rank-three-c3} and solve for $c_2$.  If
\[
  X=N+2-A,
  \qquad Y=N+2-B,
  \qquad Z=N+2-C,
\]
the result is
\begin{align*}
  c_2
  & =q^{S-N-2}[2]_q[N-2]_q![N]_q\bigl(
       [N-1]_q[X]_q[Y]_q[Z]_q\\
  &\hspace{50mm}
       -[N+2]_q[X-1]_q[Y-1]_q[Z-1]_q\bigr)\\
  & =q^{S-N-2}[2]_q[N-2]_q![N]_q
     \F_{N-1;X,Y,Z},
\end{align*}
which proves \eqref{eq:rank-three-c2}.
\end{proof}

%% file: sections/sequence_lemmas.tex
\section{Convolution and finite-window smoothing}\label{sec:sequence-lemmas}

This section isolates the two sequence-theoretic tools used throughout the
proof.  All sequences are extended by zero outside their displayed support.

\begin{lemma}[Convolution closure]\label{lem:convolution}
If two nonnegative log-concave sequences have interval support, then their
convolution is nonnegative and log-concave with interval support.
\end{lemma}

\begin{proof}
For an interval-supported nonnegative sequence $a$, adjacent log-concavity is
equivalent to total nonnegativity of order two of its lower-triangular Toeplitz
matrix
\[
  T(a)_{ij}=a_{i-j}.
\]
Indeed, on the positive support the adjacent inequalities say that the ratios
$a_j/a_{j-1}$ are weakly decreasing.  Cross-multiplication gives every
$2\times2$ Toeplitz minor; conversely, the adjacent minors give the original
inequalities, and the zero padding is harmless because the support is an
interval.
The Toeplitz matrix of a convolution is the matrix product
\[
  T(a*b)=T(a)T(b).
\]
All matrices may be truncated beyond the finitely many rows and columns
meeting a chosen minor, so the following application of Cauchy--Binet is
finite.
By Cauchy--Binet, every $2\times2$ minor of this product is a sum of products
of nonnegative $2\times2$ minors.  Thus $a*b$ is log-concave.  Its support is
the sum of two integer intervals and is therefore an interval.
\end{proof}

In particular, every product of $q$-integers is log-concave, and a monomial
factor only translates the support.

\begin{lemma}[Finite-window smoothing]\label{lem:finite-window}
Let $f=(f_0,\ldots,f_d)$ be a nonzero nonnegative palindromic unimodal
sequence with $f_0f_d>0$, and write $f(q)=\sum_{i=0}^df_iq^i$.  Put
\[
  \delta_i=f_i-f_{i-1},
  \qquad
  S_k=\sum_{i=0}^k f_i,
\]
where $f_i=0$ outside $0\leq i\leq d$.  Suppose that
$(\delta_0,\ldots,\delta_{\lfloor d/2\rfloor})$ is unimodal, with first
mode $p$.  Let $L\geq2$ be an integer.  If $p<L$ and
\begin{equation}\label{eq:prefix-hypothesis}
  S_k^2\geq S_{k-1}S_{k+1}
  \qquad\bigl(1\leq k\leq\min\{L-1,d-1\}\bigr)
\end{equation}
then $[L]_qf(q)$ is log-concave.
\end{lemma}

\begin{proof}
First suppose $L\geq d+1$, and put $T=S_d$.  On the left side of the
palindromic output, the coefficients are the prefix sums $S_k$, followed by a
constant plateau of value $T$ if $L>d+1$.  The strict prefix is covered by
\eqref{eq:prefix-hypothesis}.  If $L=d+1$, the center margin is
\[
  T^2-(T-f_d)(T-f_0)=T^2-(T-f_0)^2\geq0.
\]
If $L\geq d+2$, the prefix-to-plateau margin is
\[
  T^2-(T-f_d)T=Tf_d\geq0,
\]
and internal plateau margins vanish.  Palindromicity handles the other half.

Now assume $L\leq d$.  For $a\geq0$, define
\[
  R_a=\sum_{i=a}^{a+L-1}f_i,
  \qquad R_{-1}=S_{L-2},
\]
and set
\[
  d_a=R_a-R_{a-1},
  \qquad e_a=d_{a+1}-d_a.
\]
Cancellation at the two ends gives
\begin{equation}\label{eq:window-differences}
  d_a=f_{a+L-1}-f_{a-1},
  \qquad
  e_a=\delta_{a+L}-\delta_a.
\end{equation}
The log-concavity margin centered at $R_a$ is
\begin{equation}\label{eq:window-margin}
  M_a=R_a^2-R_{a-1}R_{a+1}=d_a^2-R_{a-1}e_a.
\end{equation}
At the prefix-to-window transition,
\begin{align}
  M_0
  &=S_{L-1}^2-S_{L-2}(S_L-f_0)\notag\\
  &=\bigl(S_{L-1}^2-S_{L-2}S_L\bigr)+f_0S_{L-2}>0.
  \label{eq:transition-margin}
\end{align}

Put
\[
  h=\left\lfloor\frac d2\right\rfloor,
  \qquad
  A=\left\lfloor\frac{d-L+1}{2}\right\rfloor.
\]
Only $0\leq a\leq A$ is needed on the left half of the output.  If
$e_b>0$, then
\[
  \delta_{b+L}>\delta_b\geq0.
\]
The full difference sequence is anti-palindromic about the coefficient
center, so positivity forces $b+L\leq h$.  Since $b+L>L>p$, unimodality of
the left-half differences rules out $b\geq p$.  Hence
\begin{equation}\label{eq:positive-e-location}
  e_b>0\quad\Longrightarrow\quad b<p\ \text{ and }\ b+L\leq h.
\end{equation}
If $0\leq a<b$ and $e_b>0$, then left-half unimodality gives
\[
  \delta_a\leq\delta_b,
  \qquad
  \delta_{a+L}\geq\delta_{b+L},
\]
because $a,b<p<a+L<b+L\leq h$.  Therefore
\begin{equation}\label{eq:positive-e-initial}
  e_a=\delta_{a+L}-\delta_a
  \geq\delta_{b+L}-\delta_b=e_b>0.
\end{equation}
Thus the positive $e_a$ form an initial interval.

While $e_a>0$, the increment at $a+1$ is nonnegative because $a<p$, and
the increment at $a+L+1$ is nonpositive on the post-mode left half.  If
$a+L=h$, the same conclusion follows from the central parity relation
$\delta_{h+1}\leq\delta_h$.  Consequently
\begin{equation}\label{eq:e-monotone}
  e_{a+1}-e_a
  =(\delta_{a+L+1}-\delta_{a+L})
   -(\delta_{a+1}-\delta_a)\leq0.
\end{equation}

If $e_a\leq0$, \eqref{eq:window-margin} gives $M_a\geq0$.  For $a<A$ in
the initial positive interval, the next moving coefficient lies before the
output center, so $2a+L\leq d$ and palindromic unimodality yields
\[
  d_{a+1}=f_{a+L}-f_a\geq0.
\]
Using \eqref{eq:e-monotone},
\begin{align}
  M_{a+1}
  &=d_{a+1}^2-R_ae_{a+1}\notag\\
  &\geq d_{a+1}^2-R_ae_a
   =M_a+e_ad_{a+1}
  \geq M_a.
  \label{eq:margin-propagation}
\end{align}
Starting from \eqref{eq:transition-margin}, this proves every moving-window
margin on the left.  Earlier margins are exactly
\eqref{eq:prefix-hypothesis}, and palindromicity reflects the result to the
right half.
\end{proof}

%% file: sections/quartic_primitive.tex
\section{The quartic kernel and its primitive}\label{sec:quartic-primitive}

Assume throughout this section that $M,x,y,z$ are integers satisfying
\begin{equation}\label{eq:primitive-domain}
  M\geq1,
  \qquad 2\leq x\leq y\leq z\leq M+2,
  \qquad (x,y,z)\neq(M+2,M+2,M+2).
\end{equation}
Set
\begin{equation}\label{eq:gap-parameters}
  a=x,
  \qquad r=y-x,
  \qquad s=z-y,
  \qquad t=M+2-z,
  \qquad R=r+s+t.
\end{equation}
Then $a\geq2$, $r,s,t\geq0$, $R\geq1$, and
\begin{equation}\label{eq:M-a-R}
  M=a+R-2.
\end{equation}

We repeatedly use the elementary transfer identity
\begin{equation}\label{eq:transfer}
  [u]_q[v]_q-[u+1]_q[v-1]_q
  =q^{v-1}[u-v+1]_q,
\end{equation}
valid for integers $v\geq1$ and $u\geq v-1$.
Indeed, after multiplying by $(1-q)^2$, the left side becomes
$(1-q)(q^{v-1}-q^u)$, which gives the right side immediately.

\begin{proposition}[Common-center decomposition]\label{prop:common-centered}
Under \eqref{eq:primitive-domain},
\begin{equation}\label{eq:F-to-G}
  \F_{M;x,y,z}=q^{a-1}G,
\end{equation}
where
\begin{align}
G={}&[a+r]_q[a+r+s]_q[R-1]_q\notag\\
&+q^r[a-1]_q[a+r+s]_q[s+t]_q\notag\\
&+q^{r+s}[a-1]_q[a+r-1]_q[t+1]_q.
\label{eq:G-positive}
\end{align}
Every nonzero summand in \eqref{eq:G-positive} is symmetric about the same
center.  Consequently, after its leading and trailing zeroes are deleted,
$\F_{M;x,y,z}$ has a nonnegative palindromic unimodal coefficient sequence
with interval support.
\end{proposition}

\begin{proof}
Insert the three intermediate products obtained by raising the first factor
and lowering, one at a time, the last three factors in
\eqref{eq:intro-quartic}.  Explicitly,
\begin{align*}
\F_{M;x,y,z}
={}&[M]_q[a]_q[a+r]_q[a+r+s]_q\\
 &-[M+1]_q[a-1]_q[a+r]_q[a+r+s]_q\\
 &+[M+1]_q[a-1]_q[a+r]_q[a+r+s]_q\\
 &-[M+2]_q[a-1]_q[a+r-1]_q[a+r+s]_q\\
 &+[M+2]_q[a-1]_q[a+r-1]_q[a+r+s]_q\\
 &-[M+3]_q[a-1]_q[a+r-1]_q[a+r+s-1]_q.
\end{align*}
Pairing consecutive lines, applying \eqref{eq:transfer}, and using
$M=a+r+s+t-2$, we obtain
\eqref{eq:F-to-G}--\eqref{eq:G-positive}.

The sum $G$ is nonzero.  Indeed, its first line is nonzero when $R\geq2$;
when $R=1$, inspection of the three possibilities for $(r,s,t)$ shows that
one of the last two lines is nonzero.

Twice the symmetry center of $q^u\prod_i[m_i]_q$ is
\[
  2u+\sum_i(m_i-1).
\]
For each line of \eqref{eq:G-positive}, this value is
\begin{equation}\label{eq:common-center}
  D=2a+3r+2s+t-4.
\end{equation}
Each nonzero line is nonnegative, palindromic, unimodal, and supported on an
integer interval.  A sum of such sequences having the same center has the
same properties.
\end{proof}

For the normalized coefficient sequence $f=(f_0,\ldots,f_d)$ of $G$, put
\begin{equation}\label{eq:primitive-delta}
  \delta_0=f_0,
  \qquad
  \delta_j=f_j-f_{j-1}
  \quad\left(1\leq j\leq\left\lfloor\frac d2\right\rfloor\right).
\end{equation}

\begin{theorem}[Primitive theorem]\label{thm:primitive}
The sequence $\delta$ in \eqref{eq:primitive-delta} is nonnegative and
unimodal.  If $p$ is its first mode, then
\begin{equation}\label{eq:primitive-mode-bound}
  p\leq R=M-x+2\leq M.
\end{equation}
\end{theorem}

\begin{proof}
Assume first that $R\geq2$.  Then $[R-1]_q\neq0$, so $G$ has nonzero
constant term and $f$ is its ordinary coefficient sequence.

The Clebsch--Gordan expansion for two $q$-integers is
\begin{equation}\label{eq:clebsch-gordan}
  [A]_q[B]_q
  =\sum_{i=0}^{\min(A,B)-1}q^i[A+B-1-2i]_q.
\end{equation}
For a shifted product $q^u[P]_q[Q]_q$, its forward difference on the left
of its symmetry center is one precisely on
\begin{equation}\label{eq:two-factor-interval}
  [u,u+\min(P,Q)-1]
\end{equation}
and is zero elsewhere.  Apply \eqref{eq:clebsch-gordan} to the two
$a$-dependent factors in each line of \eqref{eq:G-positive}.  This represents
$\delta_j$ as the number of intervals containing $j$.  The interval data are
\begin{equation}\label{eq:interval-data}
\begin{array}{c|c|c|c|c}
\text{family}&i&u_i&P_i&Q_i\\ \hline
0&0\leq i<a+r&i&2a+2r+s-1-2i&R-1\\
1&0\leq i<a-1&r+i&2a+r+s-2-2i&s+t\\
2&0\leq i<a-1&r+s+i&2a+r-3-2i&t+1.
\end{array}
\end{equation}
Rows with $Q_i=0$ are omitted, and each remaining interval is
\begin{equation}\label{eq:primitive-interval}
  I_i=[u_i,u_i+\min(P_i,Q_i)-1].
\end{equation}
Here an end event at $j$ means that the interval has right endpoint $j-1$.
Thus
\begin{equation}\label{eq:event-difference}
  \delta_j=\#\{I_i:j\in I_i\},
  \qquad
  \delta_j-\delta_{j-1}
  =\#\{\text{starts at }j\}-\#\{\text{ends at }j\}.
\end{equation}

The three start ranges are
\begin{equation}\label{eq:start-ranges}
  [0,a+r-1],
  \qquad [r,a+r-2],
  \qquad [r+s,a+r+s-2].
\end{equation}
Put
\begin{equation}\label{eq:last-start}
  T=\max\{a+r-1,a+r+s-2\}.
\end{equation}
The earliest ends supplied by a $P_i$-branch in the three families are
\begin{equation}\label{eq:P-ends}
  a+r+s,
  \qquad a+2r+s,
  \qquad a+2r+s-1,
\end{equation}
all at least $T$.  The earliest $Q_i$-branch ends are
\begin{equation}\label{eq:Q-ends}
  R-1,
  \qquad R,
  \qquad R+1.
\end{equation}
If $T<R$, no interval ends before $T$ and no interval starts after $T$.
The increments in \eqref{eq:event-difference} are therefore nonnegative
before $T$ and nonpositive after $T$.

Suppose that $R\leq T$.  No interval ends before $R-1$.  The only possible
failure of one sign change is a negative increment at $R-1$ followed by a
positive increment at $R$.  A negative increment at $R-1$ forces the
family-$0$ interval with $i=0$ to end there; it is the unique possible end at that
index, and negativity forces no interval to start there.  The family-$0$
start range in \eqref{eq:start-ranges} is contiguous and begins at zero, so
the absence of a family-$0$ start at $R-1$ also excludes one at $R$.  In
particular $R\geq2$ and $s+t>0$.  The inequality $R\leq T$ implies
\[
  s+t\leq a-1
  \quad\text{or}\quad
  t\leq a-2,
\]
and hence
\begin{equation}\label{eq:t-bound}
  t\leq a-1\leq2a+r-2.
\end{equation}
At $R$, the family-$0$ interval with $i=1$ and the family-$1$ interval with
$i=0$ both use their $Q$-branches, because in both cases
\begin{equation}\label{eq:seam-branches}
  P_i-Q_i=2a+r-t-2\geq0.
\end{equation}
Both end at $R$.  Family $0$ has no start there, and each of the other two
families contributes at most one start.  Hence
\begin{equation}\label{eq:seam-nonpositive}
  \delta_R-\delta_{R-1}\leq0.
\end{equation}

It remains to rule out a later rise.  The left primitive of a three-factor
product with $1\leq\ell_1\leq\ell_2\leq\ell_3$ is
\begin{equation}\label{eq:three-factor-primitive}
\pi_j=\max\!\left\{
0,
\min\!\left(
j+1,\ell_1,\ell_1+\ell_2-1-j,
\ell_1+\ell_2+\ell_3-2-2j
\right)
\right\}.
\end{equation}
Indeed, subtract adjacent coefficients after convolving
$[\ell_2]_q[\ell_3]_q$ with $[\ell_1]_q$; the difference of the resulting
trapezoid coefficients simplifies to \eqref{eq:three-factor-primitive}.
The formula is the minimum of one increasing and three nonincreasing
functions, truncated below at zero, so its first mode is at most
$\ell_1-1$.

Apply this to the three lines of \eqref{eq:G-positive}.  After reordering
the factors, the smallest length is no larger than $R-1$, $s+t$, and $t+1$,
respectively.  Adding the shifts $0,r,r+s$ bounds the three primitive modes
by
\begin{equation}\label{eq:line-mode-bounds}
  R-2,
  \qquad R-1,
  \qquad R.
\end{equation}
Their sum cannot rise after $R$.  Together with
\eqref{eq:seam-nonpositive}, this proves
\eqref{eq:primitive-mode-bound} for $R\geq2$.

If $R=1$, the only gap triples are
\[
  (r,s,t)=(1,0,0),(0,1,0),(0,0,1).
\]
After deleting the necessary leading shift, \eqref{eq:G-positive} becomes,
respectively,
\begin{equation}\label{eq:R-one-cores}
  [a-1]_q[a]_q,
  \qquad
  [a-1]_q[2]_q[a]_q,
  \qquad
  [a-1]_q\bigl([a]_q+[2]_q[a-1]_q\bigr).
\end{equation}
The first primitive is a string of ones.  For the second,
\eqref{eq:three-factor-primitive} gives first mode at most $1$ because one
factor is $[2]_q$.  The third primitive is $(2)$ for $a=2$ and
$(2,3,\ldots,3)$ for $a\geq3$.  Each first mode is at most $1=R$.
\end{proof}

The exclusion in \cref{thm:uniform-quartic} is sharp for the algebraic
family.  Direct expansion of $[m]_q=(1-q^m)/(1-q)$ gives
\begin{equation}\label{eq:all-max-negative}
  \F_{M;M+2,M+2,M+2}=-q^M[2M+3]_q.
\end{equation}

%% file: sections/quartic_prefix.tex
\section{Log-concavity of cumulative prefixes}\label{sec:quartic-prefix}

We now establish the second analytic input to the uniform quartic theorem.

\begin{theorem}[Cumulative-prefix theorem]\label{thm:cumulative-prefixes}
Let $M,x,y,z$ be integers satisfying
\[
  M\geq1,
  \qquad 1\leq x\leq y\leq z\leq M+2,
  \qquad x\leq M+1.
\]
After deleting the leading and trailing zeroes of
$\F_{M;x,y,z}$, write its coefficient sequence as
$f_0,\ldots,f_d$ and put
\[
  P_k=\sum_{i=0}^{k}f_i.
\]
Then
\begin{equation}\label{eq:cumulative-prefix-margin}
  P_k^2\geq P_{k-1}P_{k+1}
  \qquad(1\leq k<d).
\end{equation}
\end{theorem}

\begin{proof}
If $x=1$, the normalized kernel is $[M]_q[y]_q[z]_q$.  If
$(r,s,t)=(1,0,0)$ in the notation of \eqref{eq:gap-parameters}, then
\eqref{eq:G-positive}, after deletion of one additional leading $q$, is
$[M]_q[M+1]_q$.  Let $A(q)$ denote the corresponding product.  The
coefficients of $[d+1]_qA(q)$ through degree $d$ are
exactly $P_0,\ldots,P_d$.  This finite product is log-concave by
\cref{lem:convolution}, so \eqref{eq:cumulative-prefix-margin} follows.  We
henceforth exclude these two cases.

Put
\begin{equation}\label{eq:prefix-gaps}
  r=y-x,
  \qquad s=z-y,
  \qquad t=M+2-z,
  \qquad \tau=M-x+1.
\end{equation}
After removing the monomial $q^{x-1}$, the decomposition
\eqref{eq:G-positive} reads
\begin{align}
K={}&[z]_q[y]_q[M-x+1]_q\notag\\
 &+q^r[z]_q[x-1]_q[M-y+2]_q\notag\\
 &+q^{r+s}[y-1]_q[x-1]_q[M-z+3]_q.
\label{eq:prefix-positive-decomposition}
\end{align}
Multiplication by $1-q$ gives
\begin{equation}\label{eq:reflected-identity}
  (1-q)K=U-q^\tau V,
\end{equation}
where
\begin{align}
 U={}&[y]_q[z]_q+q^r[x-1]_q[z]_q
       +q^{r+s}[x-1]_q[y-1]_q,
 \label{eq:U-definition}\\
 V={}&[y]_q[z]_q+q[x-1]_q[z]_q
       +q^2[x-1]_q[y-1]_q.
 \label{eq:V-definition}
\end{align}
Define formal power series
\begin{equation}\label{eq:A-B-series}
  \mathcal A(q)=\frac{U(q)}{(1-q)^2},
  \qquad
  \mathcal B(q)=\frac{V(q)}{(1-q)^2}.
\end{equation}
Since $K/(1-q)$ is the generating series for the cumulative coefficients,
\begin{equation}\label{eq:prefix-reflected-series}
  \sum_{k\geq0}P_kq^k
  =\mathcal A(q)-q^\tau\mathcal B(q).
\end{equation}

\smallskip
\noindent\emph{The unreflected sequence.}
Write $u_i=[q^i]U$, $A_k=[q^k]\mathcal A$, and $A_{-1}=0$.  Set
\[
  D_k=A_k-A_{k-1}=\sum_{i=0}^{k}u_i,
  \qquad Q_k=A_k^2-A_{k-1}A_{k+1}.
\]
A direct cancellation gives, for $k\geq1$,
\begin{equation}
 Q_k-Q_{k-1}
 =D_ku_k+A_{k-1}(u_k-u_{k+1}).
 \label{eq:unreflected-margin-two}
\end{equation}
The first differences of the three summands in
\eqref{eq:U-definition} have positive-slope intervals
\begin{equation}\label{eq:U-positive-slopes}
  [0,y-1],
  \qquad [r,y-2],
  \qquad [r+s,z-2],
\end{equation}
respectively; every negative-slope interval begins at or after $z-1$.
Thus $(u_i)$ is nondecreasing through degree $z-2$ and nonincreasing from
degree $z-1$ onward.  Equation \eqref{eq:unreflected-margin-two} propagates
$Q_k\geq0$ for $k\geq z-1$.  If $z\leq2$, the initial margin
$Q_0=A_0^2$ suffices.  It remains to prove
\begin{equation}\label{eq:early-unreflected-margins}
  Q_k\geq0\qquad(1\leq k\leq z-2).
\end{equation}

\smallskip
\noindent\emph{The six-wall reduction.}
For an integer wall $W$, set
\begin{equation}\label{eq:cubic-wall}
 T_W(n)=
 \begin{cases}
  \binom{n-W+3}{3},&n\geq W,\\
  0,&n<W.
 \end{cases}
\end{equation}
If $W_i$ are the active monomial walls of $(1-q)^2U$, with signs
$\epsilon_i$, then
\begin{equation}\label{eq:wall-expansion-A}
  A_n=\sum_i\epsilon_iT_{W_i}(n).
\end{equation}
Through the degrees relevant to \eqref{eq:early-unreflected-margins}, the
baseline signed walls are
\begin{equation}\label{eq:six-baseline-walls}
  1+q^r+q^{r+s}
  -q^{x+r-1}-q^{x+r}-q^{x+r+s-1}.
\end{equation}
When $r=0$, one further negative wall can enter only at the last tested
coefficient $A_{k+1}$.  It decreases that coefficient and therefore
increases $Q_k$; it is consequently enough to analyze the six walls in
\eqref{eq:six-baseline-walls}.

Fix $k\geq1$, put $K_0=k+1$, and suppose the walls active at $A_{k+1}$
are
\[
  W_0\leq\cdots\leq W_j\leq K_0.
\]
Define
\begin{equation}\label{eq:wall-moments}
  v_i=K_0-W_i+1\geq1,
  \qquad
  \mu_p=\sum_{i=0}^{j}\epsilon_iv_i^p.
\end{equation}
For $\nu=-1,0,1$, the cubic-wall formula is
\begin{equation}\label{eq:three-wall-values}
  A_{k+\nu}=\sum_{i=0}^{j}\epsilon_i
  \binom{K_0+\nu-W_i+2}{3}.
\end{equation}
The polynomial convention in \eqref{eq:three-wall-values} remains literal
for a newly active wall, because $\binom13=\binom23=0$.  Pascal's identity
gives
\[
 A_k-A_{k-1}=\sum_i\epsilon_i\binom{v_i}{2},
 \qquad
 A_{k+1}-2A_k+A_{k-1}=\mu_1.
\]
Substitution yields the exact moment identity
\begin{equation}\label{eq:wall-moment-identity}
  12Q_k=3\mu_2^2-2\mu_1\mu_3-\mu_1^2.
\end{equation}

Every active prefix of either sign word encountered below has at least as
many positive as negative walls.  Pair each negative wall with an earlier
positive wall.  Since the walls are ordered, the corresponding values
satisfy $v_+\geq v_-$.  An unmatched positive value $v$ contributes the
integer interval $1\leq h\leq v$, while a pair contributes
$v_-+1\leq h\leq v_+$.  Let $\mathcal T$ be the resulting multiset of
positive integers and replace each $h\in\mathcal T$ by the positive odd
integer $2h-1$.  Denote the resulting multiset by $\mathcal X$ and put
\[
 \kappa=|\mathcal X|,
 \qquad S=\sum_{w\in\mathcal X}w,
 \qquad T=\sum_{w\in\mathcal X}w^2.
\]
Telescoping powers across the paired intervals gives
\[
  \mu_1=\kappa,
  \qquad \mu_2=S,
  \qquad \mu_3=\frac{3T+\kappa}{4}.
\]
Consequently,
\begin{equation}\label{eq:Psi-identity}
  8Q_k=\Psi(\mathcal X),
  \qquad
  \Psi(\mathcal X):=2S^2-\kappa T-\kappa^2.
\end{equation}

\begin{lemma}[Odd-interval inequality]\label{lem:odd-interval}
If $\mathcal X$ is the multiset union of one interval of consecutive
positive odd integers and at most two of its subintervals, then
$\Psi(\mathcal X)\geq0$.
\end{lemma}

\begin{proof}
An odd interval of length $n$ and center $c$ is
\[
  \{c-n+1,c-n+3,\ldots,c+n-1\}.
\]
Its cardinality, sum, and sum of squares are
\begin{equation}\label{eq:odd-interval-moments}
  n,
  \qquad nc,
  \qquad n\left(c^2+\frac{n^2-1}{3}\right).
\end{equation}
Translating all $\kappa$ entries by $\eta\geq0$ sends $(S,T)$ to
$(S+\kappa\eta,T+2\eta S+\kappa\eta^2)$ and therefore changes the
functional by
\begin{equation}\label{eq:Psi-translation}
  \Psi' -\Psi=2\kappa\eta S+\kappa^2\eta^2\geq0.
\end{equation}
It suffices to place the left endpoint of the outer interval at $1$.
Let its length be $n$, so its center is $n$.  Let the two subintervals
have lengths $a,b\leq n$, with length zero allowed, and centers $c_a,c_b$.
When a length is zero, choose its center arbitrarily in the displayed
interval; its value does not affect any moment.
Containment is equivalent to
\begin{equation}\label{eq:subinterval-centers}
  a\leq c_a\leq2n-a,
  \qquad b\leq c_b\leq2n-b.
\end{equation}
For fixed $c_b$, \eqref{eq:odd-interval-moments} shows that the coefficient
of $c_a^2$ in $\Psi$ is $a(a-n-b)\leq0$; the analogous statement holds
with $a$ and $b$ interchanged.  Thus $\Psi$ is separately concave in the
two centers, and its minimum on the rectangle
\eqref{eq:subinterval-centers} occurs at a corner.

Let $L$ and $R$ denote left- and right-aligned subintervals.  Direct
substitution gives
\begin{align}
 \Psi_{LR}-\Psi_{LL}
 &=4b(n-b)\bigl(2a^2+n(n-a+b)\bigr),\notag\\
 \Psi_{RL}-\Psi_{LL}
 &=4a(n-a)\bigl(2b^2+n(n-b+a)\bigr),\notag\\
 \Psi_{RR}-\Psi_{LL}
 &=4n(n+a+b)\bigl(a(n-a)+b(n-b)\bigr),
 \label{eq:corner-comparisons}
\end{align}
so the all-left corner is minimal.  Reorder $0\leq a\leq b\leq n$ and
write
\[
 b=a+d,
 \qquad n=a+d+e,
 \qquad d,e\geq0.
\]
If $p_i=a^i+b^i+n^i$, then
\begin{equation}\label{eq:Psi-LL}
 \Psi_{LL}=\frac23E,
 \qquad E=3p_2^2-2p_1p_3-p_1^2,
\end{equation}
and complete expansion gives
\begin{align}
E={}&e^4+2de^3
 +(6a^2+6ad+6d^2-1)e^2+2\mathfrak Be+\mathfrak C,
 \label{eq:E-expansion}\\
\mathfrak B={}&3a(2a^2-1)+12a^2d+9ad^2+2d(2d^2-1),\notag\\
\mathfrak C={}&9a^2(a^2-1)+12ad(2a^2-1)+24a^2d^2\notag\\
 &\quad+12ad^3+4d^2(d^2-1).
\end{align}
If $(a,d)\neq(0,0)$, the coefficient of $e^2$ is at least $5$.
Moreover, if $a\geq1$, then $2a^2-1\geq1$ and
$a^2-1\geq0$; if $d\geq1$, then $2d^2-1\geq1$ and
$d^2-1\geq0$.  These observations show term by term that
$\mathfrak B,\mathfrak C\geq0$: when one of $a,d$ is zero, every term
containing it vanishes, and the remaining displayed terms are still
nonnegative.  Hence every term of \eqref{eq:E-expansion} is nonnegative.
If $a=d=0$, then $n=e\geq1$ and
$E=e^2(e^2-1)\geq0$.  Thus $\Psi_{LL}\geq0$, completing the proof.
\end{proof}

We also need the following separated-block extension.

\begin{lemma}[Separated-block inequality]\label{lem:separated-block}
Let $1\leq a\leq b$ and $r,h\geq0$.  Form the multiset union of
\begin{equation}\label{eq:separated-blocks}
 [1,a],
 \qquad [a+h+1,a+h+r+b+1],
 \qquad [a+h+2,a+h+b+1],
\end{equation}
and replace every integer $t$ by $2t-1$.  The resulting multiset
$\mathcal X$ satisfies $\Psi(\mathcal X)\geq0$.
\end{lemma}

\begin{proof}
At $h=0$, the first two intervals concatenate to one outer interval and
the third is a subinterval, so \cref{lem:odd-interval} applies.  Increasing
$h$ translates the two high intervals but not the low one.  Their
cardinalities are $L=a$ and $H=r+2b+1\geq L$.  If $S_L,S_H$ are the
corresponding sums at $h=0$ and $\eta=2h$ is the translation in the odd
variable, direct substitution gives
\begin{equation}\label{eq:separated-translation}
 \Psi(h)-\Psi(0)
 =2\eta\bigl(2HS_L+(H-L)S_H\bigr)+\eta^2H(H-L)\geq0.
\end{equation}
\end{proof}

We now verify every active prefix of the six baseline walls.  If $s\leq x$,
cancel a coincident opposite pair when $s=x$ (or, equivalently, pad by an
empty pair).  The ordered sign word is $+++---$.  With
$v_0\geq\cdots\geq v_j$, use the pairings in
\cref{tab:plus-plus-plus}.
\begin{table}[ht]
\centering
\caption{Pairings for the wall word $+++---$.}
\label{tab:plus-plus-plus}
\begin{tabular}{c@{\qquad}lll}
\toprule
$j$&\multicolumn{3}{c}{$\mathcal T$}\\
\midrule
0&$[1,v_0]$&&\\
1&$[1,v_0]$&$[1,v_1]$&\\
2&$[1,v_0]$&$[1,v_1]$&$[1,v_2]$\\
3&$[1,v_0]$&$[1,v_1]$&$[v_3+1,v_2]$\\
4&$[1,v_0]$&$[v_4+1,v_1]$&$[v_3+1,v_2]$\\
5&$[v_5+1,v_0]$&$[v_4+1,v_1]$&$[v_3+1,v_2]$\\
\bottomrule
\end{tabular}
\end{table}
In every row, the first interval contains the other two.  If it is empty,
all three intervals are empty and $\Psi=0$; otherwise
\cref{lem:odd-interval} and \eqref{eq:Psi-identity} prove $Q_k\geq0$.

If $s>x$, put $b=x-1\geq1$ and $h=s-x\geq0$.  The ordered walls and their
signs are
\begin{equation}\label{eq:second-wall-word}
\begin{array}{c|rrrrrr}
i&0&1&2&3&4&5\\ \hline
W_i&0&r&r+b&r+b+1&r+b+1+h&r+2b+1+h\\
\epsilon_i&+&+&-&-&+&-
\end{array}
\end{equation}
and we use \cref{tab:plus-plus-minus}.
\begin{table}[ht]
\centering
\caption{Pairings for the wall word $++--+-$ in
\eqref{eq:second-wall-word}.}
\label{tab:plus-plus-minus}
\begin{tabular}{c@{\qquad}lll}
\toprule
$j$&\multicolumn{3}{c}{$\mathcal T$}\\
\midrule
0&$[1,v_0]$&&\\
1&$[1,v_0]$&$[1,v_1]$&\\
2&$[1,v_0]$&$[v_2+1,v_1]$&\\
3&$[v_3+1,v_0]$&$[v_2+1,v_1]$&\\
4&$[v_3+1,v_0]$&$[v_2+1,v_1]$&$[1,v_4]$\\
5&$[v_3+1,v_0]$&$[v_2+1,v_1]$&$[v_5+1,v_4]$\\
\bottomrule
\end{tabular}
\end{table}
Rows $0$--$3$ follow from \cref{lem:odd-interval}.  In row $4$, put
$p=v_4$.  Nonactivation of the next wall gives $1\leq p\leq b$, and the
three intervals are
\[
 [1,p],
 \quad[p+h+1,p+h+r+b+1],
 \quad[p+h+2,p+h+b+1],
\]
which is \eqref{eq:separated-blocks} with $a=p$.  In row $5$, put
$p=v_5\geq1$ and translate all three intervals down by $p$.  This gives
\[
 [1,b],
 \quad[b+h+1,r+2b+h+1],
 \quad[b+h+2,2b+h+1],
\]
which is \eqref{eq:separated-blocks} with $a=b$.  Reversing the translation
can only increase $\Psi$ by \eqref{eq:Psi-translation}.

At a tied wall position, all tied walls activate simultaneously, so a
formal cutoff inside a tied block is not a chamber.  Ties such as $r=0$ or
$h=0$, coincident cancellation, and empty paired intervals all preserve the
moment identities.  The two tables therefore exhaust the twelve possible
wall chambers.  This proves \eqref{eq:early-unreflected-margins}, and hence
\begin{equation}\label{eq:unreflected-log-concavity}
  A_k^2\geq A_{k-1}A_{k+1}\qquad(k\geq1).
\end{equation}

\smallskip
\noindent\emph{The reflected boundary and propagation.}
Let $p$ be the first mode of the primitive in \cref{thm:primitive}.  By
\eqref{eq:prefix-gaps},
\begin{equation}\label{eq:R-and-tau}
  R=r+s+t=M+2-x,
  \qquad \tau=R-1,
\end{equation}
and \cref{thm:primitive} gives $p\leq R$.  Only the boundary case $p=R$
requires a margin not already controlled by $\mathcal A$.  First-mode
minimality makes the rise at $R$ strict.  Write $v_i=[q^i]V$.  From
\eqref{eq:V-definition}, $v_0=1$ and $v_1=3$, while
\eqref{eq:reflected-identity} gives
\begin{equation}\label{eq:reflected-primitive-rise}
 \delta_R-\delta_{R-1}
 =(u_R-u_{R-1})-(v_1-v_0)
 =(u_R-u_{R-1})-2>0.
\end{equation}
Each of the three summands of $U$ contributes at most one to
$u_R-u_{R-1}$.  Hence all three contribute one, so their positive-slope
intervals in \eqref{eq:U-positive-slopes} all cover $R$.  The second interval
then gives $R\leq y-2$, or
\begin{equation}\label{eq:reflected-y-bound}
  y=x+r\geq R+2.
\end{equation}
Accordingly, the first two coefficients of
$\mathcal B=V/(1-q)^2$ are $1,5$.  The only new margin is
\begin{equation}\label{eq:reflected-boundary-margin}
 \mathcal M=(A_{R-1}-1)^2-A_{R-2}(A_R-5).
\end{equation}
Condition \eqref{eq:reflected-y-bound} places every negative wall of
\eqref{eq:six-baseline-walls} beyond degree $R$.  Put
\[
 u=s+t,
 \qquad X=A_{R-2},
 \qquad c=A_{R-1}-A_{R-2}.
\]
The three positive walls give
\begin{align}
 X&=\binom{R+1}{3}+\binom{u+1}{3}+\binom{t+1}{3},
 \label{eq:reflected-X}\\
 c&=\binom{R+1}{2}+\binom{u+1}{2}+\binom{t+1}{2}.
 \label{eq:reflected-c}
\end{align}
Since $A_R=A_{R-1}+c+R+u+t+3$,
\begin{equation}\label{eq:reflected-M-simple}
  \mathcal M=(c-1)^2-X(R+u+t).
\end{equation}
Write
\[
 \Sigma=R+u+t,
 \qquad \Theta=R^2+u^2+t^2,
 \qquad \Gamma=R^3+u^3+t^3.
\]
Then
\begin{equation}\label{eq:reflected-M-power-sums}
 12\mathcal M=3(\Theta+\Sigma-2)^2-2\Sigma(\Gamma-\Sigma).
\end{equation}
Here $R\geq u\geq t\geq0$.  Every case in which this boundary is needed has
$\Sigma\geq2$, since $\Sigma=1$ would force
$(r,s,t)=(1,0,0)$, already excluded.  Moreover,
\begin{equation}\label{eq:reflected-power-inequality}
 3\Theta-2R\Sigma
 =(R-u-t)^2+2(u^2-ut+t^2)\geq0.
\end{equation}
Since $\Gamma\leq R\Theta$ and $\Theta+\Sigma-2\geq\Theta$, we obtain
\[
 3(\Theta+\Sigma-2)^2
 \geq3\Theta^2
 \geq2R\Sigma\Theta
 \geq2\Sigma\Gamma
 \geq2\Sigma(\Gamma-\Sigma).
\]
Thus $\mathcal M\geq0$.

It remains to justify that no further margins require direct checking.
Extend $f_i=P_i=0$ to negative indices and write
\[
 \delta_i=f_i-f_{i-1},
 \qquad e_i=\delta_i-\delta_{i-1},
 \qquad \mathcal Q_k=P_k^2-P_{k-1}P_{k+1}.
\]
Direct calculation gives
\begin{align}
 \mathcal Q_k&=f_kf_{k+1}-P_k\delta_{k+1},
 \label{eq:prefix-margin-delta}\\
 \mathcal Q_{k+1}-\mathcal Q_k
 &=f_{k+1}\delta_{k+1}-P_ke_{k+2}.
 \label{eq:prefix-margin-propagation}
\end{align}
Let $m$ be the first mode of $f$.  For $p-1\leq k\leq m-1$, primitive
unimodality gives $e_{k+2}\leq0$, except possibly in the endpoint case
$m=\lfloor d/2\rfloor$ and $k=m-1$; there the same inequality is the central
parity consequence of palindromicity.  Moreover,
$f_{k+1},\delta_{k+1},P_k\geq0$.  Hence
\eqref{eq:prefix-margin-propagation} propagates every nonnegative margin.
For $k\geq m$, one has $\delta_{k+1}\leq0$, so
\eqref{eq:prefix-margin-delta} proves the margin directly.  It is therefore
enough to check $\mathcal Q_k\geq0$ for $k\leq p-1$.

If $\tau=0$, then $p\leq R=1$ and the base margin
$\mathcal Q_0=P_0^2$ suffices.  Assume $\tau\geq1$.  If
$p\leq\tau-1$, every required margin lies before the reflected term begins,
so \eqref{eq:unreflected-log-concavity} applies.  If $p=\tau$, the one
neighboring correction is favorable:
\begin{equation}\label{eq:neighboring-reflection}
 \mathcal Q_{\tau-1}
 =\bigl(A_{\tau-1}^2-A_{\tau-2}A_\tau\bigr)+A_{\tau-2}\geq0.
\end{equation}
If $p=\tau+1=R$, \eqref{eq:neighboring-reflection} and
\eqref{eq:reflected-boundary-margin} cover the last two required margins.
These cases exhaust $p\leq R=\tau+1$, and
\eqref{eq:prefix-margin-delta}--\eqref{eq:prefix-margin-propagation}
complete the proof of \eqref{eq:cumulative-prefix-margin}.
\end{proof}

%% file: sections/rank_three_completion.tex
\section{Completion of the quartic and rank-three theorems}
\label{sec:rank-three-completion}

The primitive and cumulative-prefix results now combine with finite-window
smoothing to prove \cref{thm:uniform-quartic}.

\begin{proof}[Proof of \cref{thm:uniform-quartic}]
Suppose first that $x=1$.  Then $[x-1]_q=0$, so
\[
  [L]_q\F_{M;1,y,z}
  =[L]_q[M]_q[y]_q[z]_q.
\]
This is a product of $q$-integers and is therefore log-concave with interval
support by \cref{lem:convolution}.

Now assume $x\geq2$.  The excluded all-maximal triple is the only sorted
triple in the domain with $x=M+2$.  Hence $x\leq M+1$, so both
\cref{thm:primitive,thm:cumulative-prefixes} apply.  Delete the monomial
factor in \eqref{eq:F-to-G} and write the normalized coefficient sequence
as $f=(f_0,\ldots,f_d)$.  By
\cref{prop:common-centered}, this sequence is nonnegative, palindromic,
unimodal, and positive throughout its support.  By \cref{thm:primitive},
its left-half difference sequence is nonnegative and unimodal, and its first
mode $p$ satisfies
\[
  p\leq M.
\]
By \cref{thm:cumulative-prefixes}, every cumulative-prefix margin of $f$
required in \cref{lem:finite-window} is nonnegative.  Since
$L\geq M+1$, one has
\[
  p\leq M<L.
\]
The finite-window smoothing lemma therefore proves that $[L]_qf(q)$ is
log-concave.

Restoring the deleted monomial only translates the coefficient sequence.
Moreover, $f$ is positive at every index of its support and $[L]_q$ has an
all-ones coefficient sequence, so their convolution is positive at every
index between its two support endpoints.  Thus the output has interval
support, completing the proof.
\end{proof}

We return to the coefficient formulas in
\cref{prop:rank-three-coefficients}.

\begin{theorem}[Genuine rank three]
\label{thm:rank-three}
Let $h$ be an abelian Hessenberg function whose complement-Ferrers partition
$\lambda$ satisfies $\rho(\lambda)=3$.  Then every nonzero elementary-basis
coefficient of $X_{G_h}(\xx;q)$ is log-concave with interval support.
\end{theorem}

\begin{proof}
By transpose symmetry, it is enough to use the width-three orientation in
\cref{prop:rank-three-domain}.

Suppose first that the graph is connected, so the parameters satisfy
\eqref{eq:rank-three-domain}.  The support theorem and
\cref{prop:rank-three-coefficients} show that the only coefficients to
consider are $c_0,c_1,c_2,c_3$.

The formulas \eqref{eq:rank-three-c0} and
\eqref{eq:rank-three-c3} are monomial shifts of products of $q$-integers.
All of their displayed factors have nonnegative lengths on
\eqref{eq:rank-three-domain}; if a factor is $[0]_q$, the corresponding
coefficient is zero and is omitted.  Every nonzero such product is
log-concave with interval support by \cref{lem:convolution}.

For $c_1$, let $x\leq y\leq z$ be the increasing rearrangement of
$(A,B,C)$.  The kernel $\F_{N;A,B,C}$ is symmetric in these parameters, and
\eqref{eq:rank-three-domain} gives
\[
  1\leq x\leq y\leq z\leq N+1<N+2.
\]
Thus the exceptional all-maximal triple cannot occur.  Apply
\cref{thm:uniform-quartic} with
\[
  M=N,
  \qquad L=N+2=M+2.
\]
It follows that $[N+2]_q\F_{N;A,B,C}$ is log-concave with interval support.
Multiplication by the remaining factor $q[N-1]_q!$ in
\eqref{eq:rank-three-c1} preserves these properties by
\cref{lem:convolution}.

For $c_2$, put
\begin{equation}\label{eq:completion-dual-parameters}
  X=N+2-A,
  \qquad Y=N+2-B,
  \qquad Z=N+2-C.
\end{equation}
Let $x\leq y\leq z$ be the increasing rearrangement of $(X,Y,Z)$.
Equation \eqref{eq:rank-three-domain} gives
\[
  1\leq x\leq y\leq z\leq N+1.
\]
Here the quartic parameter is $M=N-1$, so $M+2=N+1$, and the smoothing
length in \eqref{eq:rank-three-c2} is
\[
  L=N=M+1.
\]
The all-maximal triple is impossible because
$X=N+2-A\leq N=M+1$.  Put
\[
  \mathcal R(q)=[N]_q\F_{N-1;X,Y,Z}(q),
  \qquad H(q)=[2]_q[N-2]_q!,
  \qquad \eta=S-N-2.
\]
The uniform quartic theorem shows that $\mathcal R(q)$ is nonzero and
log-concave with interval support, while \eqref{eq:rank-three-c2} gives
$c_2(q)=q^\eta H(q)\mathcal R(q)$.  Let $m=\max\{A,B,C\}$, so the smallest
dual parameter is $x=N+2-m$.  If $x=1$, the first case in the proof of
\cref{thm:uniform-quartic} applies; if $x\geq2$,
\cref{prop:common-centered} applies.  In either case,
$\mathcal R(q)$ is divisible by $q^{x-1}$.  Moreover,
\[
  \eta+x-1=S-m-1\geq0,
\]
because the other two members of $\{A,B,C\}$ have sum at least $2$.
Therefore $q^\eta\mathcal R(q)$ is a polynomial, and multiplication by
$q^\eta$ only translates its nonzero coefficient sequence.  Convolution
with $H(q)$ now proves the assertion for $c_2$ by
\cref{lem:convolution}.

It remains to treat the disconnected orientation.  By
\cref{prop:rank-three-domain}, it is uniquely
\[
  h=(3,3,3,n,\ldots,n).
\]
In that case
\[
  G_h=K_3\sqcup K_{n-3}.
\]
Chromatic quasisymmetric functions multiply over disjoint unions, and
$X_{K_m}(\xx;q)=[m]_q!e_m(\xx)$.  Therefore
\begin{equation}\label{eq:rank-three-disconnected-xg}
  X_{G_h}(\xx;q)
  =[3]_q![n-3]_q!e_{(n-3,3)}(\xx),
\end{equation}
whose sole coefficient is again a product of $q$-integers.  This completes
the rank-three proof.
\end{proof}

%% file: sections/lower_rank.tex
\section{Complement-Ferrers ranks zero, one, and two}
\label{sec:lower-rank}

By transpose symmetry, it is enough to consider Hessenberg functions having
at most two entries smaller than $n$.

\subsection{Ranks zero and one}

At rank zero the graph is complete, and the standard inversion enumerator
gives
\begin{equation}\label{eq:rank-zero}
  X_{K_n}(\xx;q)=[n]_q!e_n(\xx).
\end{equation}
This coefficient is a product of $q$-integers.

Now let $h=(u,n,\ldots,n)$ and put $A=u-1$.  The support statement in
\cref{prop:abelian-support} leaves at most two coefficients.  Evaluating the
specialization in \cref{prop:abelian-specialization} at $\xi=0$ and $\xi=1$
gives them explicitly:
\begin{align}
  \coeff{e_n}{X_{G_h}(\xx;q)}
    &=[A]_q[n]_q[n-2]_q!,
  \label{eq:rank-one-c0}\\
  \coeff{e_{(n-1,1)}}{X_{G_h}(\xx;q)}
    &=q^A[n-2]_q![n-1-A]_q.
  \label{eq:rank-one-c1}
\end{align}
Every nonzero coefficient in
\eqref{eq:rank-one-c0}--\eqref{eq:rank-one-c1} is a monomial shift of a
product of $q$-integers, so \cref{lem:convolution} proves the rank-one case.

\subsection{The exact rank-two coefficients}

It remains to treat functions with two entries smaller than $n$; the
parametrization also retains the boundary where the rank drops to one.
After transpose, every such case of order $n\geq4$ has
\begin{equation}\label{eq:rank-two-h}
  h=(u,v,n,\ldots,n),
  \qquad
  2\leq u\leq v\leq n-1.
\end{equation}
Set
\begin{equation}\label{eq:rank-two-parameters}
  N=n-3,
  \qquad A=u-1,
  \qquad B=v-2.
\end{equation}
The parameter inequalities needed below are
\begin{equation}\label{eq:rank-two-domain}
  1\leq A\leq B+1,
  \qquad 0\leq B\leq N.
\end{equation}

If $B=0$, the domain forces $A=1$; then $h=(2,2,n,\ldots,n)$ and
\begin{equation}\label{eq:rank-two-disconnected}
  X_{G_h}(\xx;q)=[2]_q![n-2]_q!e_{(n-2,2)}(\xx).
\end{equation}
Henceforth assume $B\geq1$.

Only $e_n,e_{(n-1,1)},e_{(n-2,2)}$ can occur.  Cancelling
$\varphi_{n-2}(\xi)=\varphi_{N+1}(\xi)$ from the specialization identity gives
\begin{align}
  (\xi-[A]_q)(\xi-[B]_q)
  ={}&c_0\frac{(\xi-[N+1]_q)(\xi-[N+2]_q)}{[n]_q!}\notag\\
    &+c_1\frac{\xi(\xi-[N+1]_q)}{[n-1]_q!}
     +c_2\frac{\xi(\xi-1)}{[n-2]_q![2]_q},
  \label{eq:rank-two-quadratic}
\end{align}
where $c_j=\coeff{e_{(n-j,j)}}{X_{G_h}(\xx;q)}$.  Evaluation at
$\xi=0,1,[N+1]_q$, in that order, yields
\begin{align}
  c_0={}&[A]_q[B]_q[n]_q[N]_q!,
  \label{eq:rank-two-c0}\\
  c_1={}&q[N-1]_q![N+1]_q F_{N,A,B}(q),
  \label{eq:rank-two-c1}\\
  c_2={}&q^{A+B-1}[2]_q[N-1]_q!
          [N+1-A]_q[N+1-B]_q,
  \label{eq:rank-two-c2}
\end{align}
where
\begin{equation}\label{eq:rank-two-kernel}
  F_{N,A,B}
  =[N]_q[A]_q[B]_q-[N+2]_q[A-1]_q[B-1]_q.
\end{equation}
The formulas follow from \eqref{eq:q-integer-evaluation-identities}: at
$\xi=0$, $1$, and $[N+1]_q$, respectively, the evaluation system is
triangular in $c_0,c_1,c_2$.

\begin{theorem}[Ranks at most two]\label{thm:rank-at-most-two}
If $h$ is abelian and its complement-Ferrers partition $\lambda$ satisfies
$\rho(\lambda)\leq2$, then every nonzero elementary-basis coefficient of
$X_{G_h}(\xx;q)$ is log-concave with interval support.
\end{theorem}

\begin{proof}
Ranks zero and one were proved in
\eqref{eq:rank-zero}--\eqref{eq:rank-one-c1}.  At rank two,
\eqref{eq:rank-two-disconnected} handles $B=0$.  For $B\geq1$, the
coefficients in \eqref{eq:rank-two-c0} and
\eqref{eq:rank-two-c2} are monomial shifts of products of $q$-integers.

It remains to consider $c_1$.  If $N=1$, then $B=1$ and
$A\in\{1,2\}$.  Directly,
\[
  F_{1,1,1}=1,
  \qquad
  F_{1,2,1}=[2]_q,
\]
so \eqref{eq:rank-two-c1} gives $c_1=q[2]_q$ or $q[2]_q^2$.
Now let $N\geq2$.  By \eqref{eq:intro-quartic},
\begin{align}
  \F_{N-1;A,B,N}
  &=[N-1]_q[A]_q[B]_q[N]_q
    -[N+2]_q[A-1]_q[B-1]_q[N-1]_q\notag\\
  &=[N-1]_qF_{N,A,B}.
  \label{eq:rank-two-quartic-specialization}
\end{align}
Let $x\leq y\leq z$ be the increasing rearrangement of $(A,B,N)$.  Since
$\F$ is symmetric in its last three parameters, it suffices to use this
sorted triple, which satisfies
\[
  1\leq x\leq y\leq z\leq N+1=(N-1)+2.
\]
Indeed, $1\leq A\leq B+1\leq N+1$ and $1\leq B\leq N$.
The triple is not all-maximal because it contains the entry $N<N+1$.
Thus \cref{thm:uniform-quartic}, with $M=N-1$ and $L=N+1=M+2$,
proves that $[N+1]_q\F_{N-1;A,B,N}$ is log-concave with interval support.
Using $[N-1]_q!=[N-2]_q![N-1]_q$ in
\eqref{eq:rank-two-c1} gives
\[
  c_1=q[N-2]_q![N+1]_q\F_{N-1;A,B,N}.
\]
The quartic is nonzero by \cref{thm:uniform-quartic}, and
\cref{lem:convolution} completes the rank-two case.

The rank-zero and rank-one formulas cover the remaining orders below $4$;
the only order-$3$ case with two entries smaller than $3$ is $h=(2,2,3)$,
already covered by \eqref{eq:rank-two-disconnected}.  Finally,
transpose symmetry covers the orientations with at most two rows rather
than at most two columns.
\end{proof}

\begin{proof}[Proof of \cref{thm:main}]
The case $\rho(\lambda)=3$ is \cref{thm:rank-three}, and the cases
$\rho(\lambda)\leq2$ are \cref{thm:rank-at-most-two}.  These cases exhaust
$\rho(\lambda)\leq3$, proving the theorem.
\end{proof}

%% file: sections/counterexample.tex
\section{The unrestricted conjecture fails}\label{sec:counterexample}

\begin{proof}[Proof of \cref{thm:intro-counterexample}]
The tuple
\[
  h=(2,4,4,6,7,10,10,10,10,12,12,13,13)
\]
is weakly increasing and satisfies $h(i)\geq i$, so it is a Hessenberg
function.  Since $h(i)\geq i+1$ for $1\leq i<13$, every consecutive edge
$\{i,i+1\}$ belongs to $G_h$; hence $G_h$ is connected.

By \cref{prop:counterexample-certificate}, the coefficient identity in
\cref{thm:intro-counterexample} holds.  If that coefficient is written
$\sum_j a_jq^j$, then
\[
  a_6^2=6^2=36<38=1\cdot38=a_5a_7.
\]
Thus it is not log-concave.  Finally, the graph is not abelian, because
\[
  h(h(1)+1)=h(3)=4\neq13.
\]
It therefore does not conflict with \cref{thm:main}.
\end{proof}

Abreu and Nigro reported exhaustive verification through order $12$
\cite[Section~5]{AbreuNigro}.  The counterexample occurs at the first order
beyond that reported computation.  We do not need, and do not claim, an
independent proof that order $13$ is minimal among all natural unit interval
graphs.

%% file: sections/verification_outlook.tex
\section{Consequences and open problems}
\label{sec:verification-outlook}

The two main results separate the false unrestricted conjecture from a
uniform positive regime: complement-Ferrers rank at most three forces
coefficientwise $e$-log-concavity in the abelian family, whereas the
counterexample is nonabelian.  We do not suggest that every counterexample
must be nonabelian.

Complement-Ferrers rank four is the first abelian layer not covered by the
present method.  This is an open boundary, not a failure boundary: we know of
no counterexample of complement-Ferrers rank four.  Rank-three interpolation is
cubic and produces the
four-factor difference in \eqref{eq:intro-quartic}; the next layer should
produce higher interpolation kernels for which the common-center and wall
methods may retain useful structure.  Two concrete problems are therefore:
\begin{enumerate}
 \item determine whether every elementary coefficient is log-concave for all
 abelian Hessenberg functions; and
 \item identify graph-theoretic conditions on nonabelian natural unit
 interval graphs that force log-concavity or systematically permit failure.
\end{enumerate}
The algebraic domain already matters for the quartic family:
\eqref{eq:all-max-negative} shows that the all-maximal quartic is negative,
while
\cref{prop:rank-three-domain} excludes its all-maximal tuple from the
graph-derived parameters.

\subsection*{Verification, provenance, and code}

The proofs of \cref{thm:uniform-quartic,thm:rank-three,thm:main} are symbolic
and uniform.  The included rank-three verifier independently checks the
coefficient formulas through $N=18$, exact interpolation through $N=10$,
modular interpolation through $N=60$, every rank-two case through $N=40$,
and direct coloring reconstructions through order $7$.  A second
program verifies the counterexample from proper colorings and an exact basis
conversion.  Both are deterministic, use only standard-library integer or
rational arithmetic, and require CPython~3.10 or later.  The arXiv source
bundle includes the programs, a reproducibility guide, and SHA-256 digests.

\ifjournalversion
\else
The counterexample, its recurrence, and its exact certificate first appeared
in version~1 of this article.  They are adapted and reproduced here so that
the negative result and the positive abelian theorem can be evaluated
together; version~1 contains no form of the rank-at-most-three theorem.
\fi

\subsection*{Use of generative-AI tools}

OpenAI's ChatGPT with Codex assisted with exploration, verification code,
literature searches, proof organization, drafting, and editing.  Mathematical
claims rest on the arguments and exact certificate presented here, and the
author remains responsible for the statements, computations, citations, and
submission.

%% file: sections/appendix_counterexample.tex
\section{Exact certificate for the counterexample}
\label{app:counterexample-certificate}

This appendix gives the mathematical specification of a fully reproducible
exact certificate for the coefficient in \cref{thm:intro-counterexample};
the accompanying program executes the calculation.  The calculation starts
directly from proper colorings and then performs an exact change of basis.  No
specialized symmetric-function software, random choice, or floating-point
arithmetic is involved.

\subsection{Fixed-content proper colorings}

Fix a composition $\mathbf r\vDash n$, written
$\mathbf r=(r_1,\ldots,r_s)$.  Its fixed-content monomial coefficient is
\begin{equation}\label{eq:fixed-monomial-coefficient}
 M_{\mathbf r}(q)
 =\coeff{x_1^{r_1}\cdots x_s^{r_s}}{X_{G_h}(\xx;q)}.
\end{equation}
For disjoint vertex sets $I,J\subseteq[n]$, define
\begin{equation}\label{eq:crossing-statistic}
 \operatorname{cr}_{G_h}(I,J)
 =\#\{\{u,v\}\in E(G_h):u<v,\ u\in I,\ v\in J\}.
\end{equation}
If $W\subseteq[n]$ and $\mathbf r\vDash|W|$, write
$\mathsf C(W;\mathbf r)$ for the ascent enumerator of proper colorings of
$G_h[W]$ whose successive color-class sizes are $\mathbf r$.

\begin{lemma}[Fixed-content recurrence]\label{lem:fixed-content-recurrence}
If $\mathbf r=(r,\mathbf s)$, then
\begin{equation}\label{eq:coloring-recurrence}
 \mathsf C(W;(r,\mathbf s))
 =\sum_{\substack{I\subseteq W,\ |I|=r\\I\ \mathrm{independent}}}
 q^{\operatorname{cr}_{G_h}(I,W\setminus I)}
 \mathsf C(W\setminus I;\mathbf s),
\end{equation}
with $\mathsf C(\varnothing;\varnothing)=1$ and
$\mathsf C(W;\varnothing)=0$ for $W\neq\varnothing$.  Moreover,
$M_{\mathbf r}(q)=\mathsf C([n];\mathbf r)$.
\end{lemma}

\begin{proof}
The set $I$ is exactly the class receiving the least remaining color, and
properness is equivalent to its independence.  Every vertex in
$W\setminus I$ receives a larger color.  The ascents determined at this
stage are therefore precisely the edges counted by
$\operatorname{cr}_{G_h}(I,W\setminus I)$.  Removing $I$ leaves the same
problem for the remaining colors.  The least color class determines a
unique summand, so \eqref{eq:coloring-recurrence} is exhaustive and has no
overcounting.  Taking $W=[n]$ gives \eqref{eq:fixed-monomial-coefficient}.
\end{proof}

Although symmetry of $X_{G_h}(\xx;q)$ is known abstractly, the certificate
does not need to assume it.  For any $i_1<\cdots<i_s$, order-preserving
relabeling of the colors identifies the coefficient of
$x_{i_1}^{r_1}\cdots x_{i_s}^{r_s}$ with $M_{\mathbf r}(q)$.  The program
evaluates all $2^{12}=4096$ compositions of $13$ and verifies equality
whenever two compositions have the same decreasing rearrangement.  This
proves symmetry and identifies the common values as the coefficients of the
$101$ monomial symmetric functions $m_\nu$ indexed by partitions
$\nu\vdash13$.  Denote this common value by $M_\nu(q)$.

\subsection{Exact conversion from the monomial basis}

For partitions $\nu,\mu\vdash13$, set
\begin{equation}\label{eq:incidence-matrix}
 \mathsf M_{\nu,\mu}
 =\coeff{x_1^{\nu_1}\cdots x_{\ell(\nu)}^{\nu_{\ell(\nu)}}}{e_\mu}.
\end{equation}

\begin{lemma}[Zero-one matrix interpretation]\label{lem:zero-one-matrix}
The integer $\mathsf M_{\nu,\mu}$ equals the number of zero-one matrices
with labeled rows and columns, row-sum sequence $\mu$, and column-sum
sequence $\nu$.  Consequently, if
$X_{G_h}(\xx;q)=\sum_{\mu\vdash13}c_\mu(q)e_\mu(\xx)$, then
\begin{equation}\label{eq:monomial-elementary-system}
 M_\nu(q)=\sum_{\mu\vdash13}\mathsf M_{\nu,\mu}c_\mu(q).
\end{equation}
\end{lemma}

\begin{proof}
In $e_\mu=\prod_i e_{\mu_i}$, row $i$ records the distinct variables
selected from the factor $e_{\mu_i}$.  Its row sum is $\mu_i$; column $j$
records the exponent of $x_j$ and has sum $\nu_j$.  This is a bijection
between choices contributing to the monomial in \eqref{eq:incidence-matrix}
and the asserted matrices.  Comparing monomial coefficients proves
\eqref{eq:monomial-elementary-system}.
\end{proof}

The matrix $\mathsf M$ is the transition matrix between two bases of the
degree-$13$ symmetric functions and is therefore invertible.  Each entry is
also computed by a finite recurrence: choose the columns occupied by the
first row, reduce their remaining capacities by one, and repeat.  Thus the
right side of \eqref{eq:monomial-elementary-system} is recoverable using
only subset enumeration and integer arithmetic.

\subsection{Certificate output and correctness}

\begin{proposition}\label{prop:counterexample-certificate}
Applying \cref{lem:fixed-content-recurrence,lem:zero-one-matrix} to
\[
 h^\star=(2,4,4,6,7,10,10,10,10,12,12,13,13)
\]
and solving \eqref{eq:monomial-elementary-system} over the rational numbers
gives the coefficient stated in \cref{thm:intro-counterexample}.
\end{proposition}

\begin{proof}
The ancillary program \texttt{anc/verify\_counterexample.py} implements the
two recurrences literally.  It enumerates every independent subset of
$[13]$, evaluates all $4096$ ordered compositions, constructs all entries of
$\mathsf M$, and performs Gauss--Jordan elimination using exact rational
numbers.  It verifies integrality of every recovered elementary coefficient.
For the target coefficient, the vector indexed from $q^0$ is
\begin{equation}\label{eq:counterexample-vector}
 (0,0,0,0,0,1,6,38,128,257,362,400,362,257,128,38,6,1).
\end{equation}
After padding through degree $|E(G_{h^\star})|=22$, the program verifies
palindromicity, positivity on the support, unimodality, and exactly the two
failed log-concavity margins
\[
 (j,a_j^2,a_{j-1}a_{j+1})=(6,36,38),(16,36,38).
\]

The program uses only exact arithmetic from Python's standard library and
reads no external data.  The ancillary documentation records its source
digest, tested software version, command, and expected output.  Because the
program implements the two proved recurrences and returns
\eqref{eq:counterexample-vector}, the claimed coefficient follows.
\end{proof}